\documentclass[11pt, letterpaper]{amsart}

\usepackage{amsmath, amssymb, amsthm, mathtools}
\usepackage{graphicx}
\usepackage{xcolor}
\usepackage{float}
\usepackage{placeins}
\usepackage{subcaption}
\usepackage{enumitem}
\usepackage{hyperref}

\theoremstyle{plain}
\newtheorem{theorem}{Theorem}[section]
\newtheorem{lemma}[theorem]{Lemma}
\newtheorem{proposition}[theorem]{Proposition}
\newtheorem{corollary}[theorem]{Corollary}

\theoremstyle{definition}
\newtheorem{definition}[theorem]{Definition}
\newtheorem{example}[theorem]{Example}

\theoremstyle{remark}
\newtheorem{remark}[theorem]{Remark}

\title[Temporal Metric Spaces and TIFS]{Temporal Metric Spaces and Temporal Iterated Function Systems}

\author{Amal P. S.}
\address{APJ Abdul Kalam Technological University, Thiruvananthapuram, Kerala, India; Rajagiri School of Engineering and Technology, Kochi, Kerala, India}
\email{amalb9111@gmail.com}
\thanks{Corresponding author: Amal P. S}

\author{Vinod Kumar P. B.}
\address{APJ Abdul Kalam Technological University, Thiruvananthapuram, Kerala, India; Muthoot Institute of Technology and Science, Kochi, Kerala, India}
\email{vinodkumarpb@mgits.ac.in}

\author{Ramkumar P. B.}
\address{APJ Abdul Kalam Technological University, Thiruvananthapuram, Kerala, India; Rajagiri School of Engineering and Technology, Kochi, Kerala, India}
\email{ramkumar\_pb@rajagiritech.edu.in}


\begin{document}

\begin{abstract}
We introduce temporal metric spaces, a framework for formalizing metric and topological structures on continuously evolving domains. Rather than defining a single metric on a fixed space, we construct a family of interval-dependent metrics on an associated trajectory space, enabling distances to be measured between continuous trajectories over finite time intervals. We establish the fundamental topology of the trajectory space, prove that it inherits completeness from the underlying metric space, and develop the corresponding temporal hyperspace together with its completeness. Finally, we develop a theory of iterated function systems acting on these evolving spaces, proving that their attractors evolve coherently with the underlying dynamics and remain stable even when the evolving geometry causes the generating mappings to lose their classical contraction property.
\end{abstract}

\maketitle

\section{Introduction}

Since Fréchet formalized the notion of a metric space \cite{frechet1906} and Banach established the contraction mapping principle \cite{banach1922}, metric methods have become one of the central tools in analysis, topology, and dynamical systems. Motivated by both theoretical developments and practical applications, numerous generalizations of metric spaces have subsequently been proposed. These include partial metric spaces, in which the self-distance of a point need not vanish \cite{matthews1994}, cone metric spaces, where distances take values in ordered Banach spaces rather than the real line \cite{huang2007}, and various generalized metric structures that evaluate three or more points simultaneously by defining the distance function on $X^n$ rather than the classical $X \times X$. A comprehensive survey of these and related generalizations is given by Khamsi \cite{khamsi2015}. Despite their diversity, these theories all preserve a common underlying assumption: the ambient set on which the metric is defined remains fixed.

These metric structures have provided the natural setting for the development of fractal geometry through iterated function systems (IFSs). In the classical theory, an IFS consists of finitely many contractions acting on a complete metric space, and Hutchinson's theorem guarantees the existence of a unique compact attractor \cite{hutchinson1981}. This framework was further developed by Barnsley and Demko \cite{barnsley1985} and has become a cornerstone of modern fractal geometry \cite{falconer2014}. Numerous extensions have subsequently been introduced by relaxing the classical contraction conditions or by allowing the generating mappings to vary over time. Examples include non-autonomous iterated function systems, whose generating maps vary with time while still acting on a fixed metric space \cite{rempe2016}, Kannan iterated function systems \cite{sahu2010}, and more recently $G$-iterated function systems, where contractivity is achieved only after finite compositions of the generating maps \cite{amal2026}.

Despite these substantial advances, virtually all existing frameworks share a fundamental limitation: the underlying domain itself is assumed to remain static. The metric may be generalized, the contraction condition may be weakened, and the generating maps may vary with time, yet the ambient space on which the dynamics occur is fixed. This assumption breaks down when modeling systems whose domains continuously evolve, such as spaces undergoing expansion, contraction, shear, or other continuous deformations. In such dynamic settings, the generating mappings of the IFS evolve concurrently with the space itself. As the domain deforms over time, these evolving mappings may fail to remain strict contractions. Consequently, the hypotheses required by classical fixed-point and iterated function system theory no longer hold.

To overcome this limitation, we introduce the notion of a \emph{temporal metric space}. Rather than forcing dynamics to occur on a single static domain, we analyze systems whose domains deform continuously according to an evolution family. We embed these moving domains into a trajectory space equipped with temporal metrics that evaluate the distance between continuous trajectories over intervals of time. This framework allows the domain to deform while preserving the intrinsic topology of the underlying dynamics.

The paper is organized as follows. In Section 2, we introduce the concept of temporal metric spaces by formalizing continuously evolving domains and constructing interval-dependent metrics on their associated trajectory spaces, concluding with several examples. In Section 3, we establish the fundamental topological structure of the trajectory space, analyzing convergence conditions, sequential compactness, and the continuous lifting of dynamically equivariant mappings. In Section 4, we construct the temporal hyperspace equipped with a temporal Hausdorff metric and formally prove its completeness. Finally, in Section 5, we develop the theory of temporal iterated function systems, providing explicit examples and proving that flow-induced attractors maintain their topological stability even when time-dependent evolution causes the generating mappings to lose their classical contraction properties.

\section{Temporal Metric Spaces}

\begin{definition}
Let $(X, d)$ be a metric space and $T \subseteq \mathbb{R}$ an interval. A \textit{temporal metric space} over $X$ is a 4-tuple $(\mathcal{X}, \Phi, \Gamma, \mathcal{D})$ satisfying the following conditions:

\begin{enumerate}
    \item A family of \textit{fibers} $\mathcal{X} = \{X(t)\}_{t \in T}$, where $X(t) \subseteq X$ for all $t \in T$.
    
    \item An \textit{evolution family} $\Phi = \{\Phi_{t, s}: X(s) \to X(t) \mid s, t \in T\}$ consisting of homeomorphisms such that:
    \begin{enumerate}
        \item[(a)] $\Phi_{t, t} = \mathrm{Id}_{X(t)}$ for all $t \in T$.
        \item[(b)] $\Phi_{t, s} \circ \Phi_{s, r} = \Phi_{t, r}$ for all $r, s, t \in T$.
        \item[(c)] For each fixed $s \in T$, the map $T \times X(s) \to X$ given by $(t, x) \mapsto \Phi_{t, s}(x)$ is continuous.
    \end{enumerate}
    
    \item A \textit{trajectory space} $\Gamma = \{\gamma \in C(T, X) \mid \gamma(t) = \Phi_{t, s}(\gamma(s)) \text{ for all } s, t \in T\}$.
    
    \item A family of \textit{temporal metrics} $\mathcal{D} = \{\mathcal{D}_{[a, b]}\}$ indexed by the closed subintervals $[a, b] \subseteq T$, such that for every $[a, b]$, the map $\mathcal{D}_{[a, b]}: \Gamma \times \Gamma \to [0, \infty)$ is a metric on $\Gamma$. For any $\gamma_1, \gamma_2 \in \Gamma$, the mapping $(a, b) \mapsto \mathcal{D}_{[a, b]}(\gamma_1, \gamma_2)$ is continuous on the domain $\{(a, b) \in T \times T \mid a \leq b\}$. Furthermore, for any $t \in T$, the metric over the degenerate interval $[t, t]$ satisfies:
    \[
        \mathcal{D}_{[t, t]}(\gamma_1, \gamma_2) = d(\gamma_1(t), \gamma_2(t)).
    \]
\end{enumerate}
\end{definition}

\begin{lemma}
\label{lem:global_identity}
Let $(\mathcal{X}, \Phi, \Gamma, \mathcal{D})$ be a temporal metric space. For any $\gamma_1, \gamma_2 \in \Gamma$, if there exists a time $t^* \in T$ such that $\gamma_1(t^*) = \gamma_2(t^*)$, then $\gamma_1 = \gamma_2$.
\end{lemma}

\begin{proof}
Let $t \in T$. By the definition of $\Gamma$, we have $\gamma(t) = \Phi_{t, t^*}(\gamma(t^*))$ for any trajectory $\gamma \in \Gamma$. Applying this to $\gamma_1$ and $\gamma_2$ along with the hypothesis $\gamma_1(t^*) = \gamma_2(t^*)$ yields
\[
    \gamma_1(t) = \Phi_{t, t^*}(\gamma_1(t^*)) = \Phi_{t, t^*}(\gamma_2(t^*)) = \gamma_2(t).
\]
Since $t \in T$ was arbitrary, we conclude that $\gamma_1 = \gamma_2$.
\end{proof}

We now present three concrete examples of temporal metric spaces.

\begin{example}
Let $(X, d)$ be a metric space and $T \subseteq \mathbb{R}$. Define the static temporal metric space $(\mathcal{X}, \Phi, \Gamma, \mathcal{D})$ over $X$ by:
\begin{enumerate}
    \item $X(t) = X$ for all $t \in T$.
    \item $\Phi_{t, s}(x) = x$ for all $x \in X$ and $s, t \in T$.
    \item $\Gamma = \{\gamma \in C(T, X) \mid \gamma(t) = x_0 \text{ for some fixed } x_0 \in X\}$.
    \item For any $\gamma_1, \gamma_2 \in \Gamma$ identically equal to $x_1, x_2 \in X$, set $\mathcal{D}_{[a, b]}(\gamma_1, \gamma_2) = d(x_1, x_2)$ for all $a \leq b$.
\end{enumerate}

It is immediate that $\Phi_{t, t} = \mathrm{Id}_{X}$ and $\Phi_{t, s} \circ \Phi_{s, r} = \Phi_{t, r}$. Continuity of the map $(t, x) \mapsto x$ is trivial. The metric mapping $(a, b) \mapsto d(x_1, x_2)$ is constant, hence continuous. Over the degenerate interval, $\mathcal{D}_{[t, t]}(\gamma_1, \gamma_2) = d(x_1, x_2) = d(\gamma_1(t), \gamma_2(t))$.
\end{example}

\begin{example}
Let $X = \mathbb{R}$ with the Euclidean metric $d(x, y) = |x - y|$, and $T = [1, \infty)$. Define $(\mathcal{X}, \Phi, \Gamma, \mathcal{D})$ by:
\begin{enumerate}
    \item $X(t) = [-t, t] \subset \mathbb{R}$ for all $t \in T$.
    \item $\Phi_{t, s}(x) = (t/s)x$ for all $x \in X(s)$ and $s, t \in T$.
    \item $\Gamma = \{\gamma \in C(T, X) \mid \gamma(t) = t x_0 \text{ for some } x_0 \in [-1, 1]\}$.
    \item $\mathcal{D}_{[a, b]}(\gamma_1, \gamma_2) = \max_{\tau \in [a, b]} |\gamma_1(\tau) - \gamma_2(\tau)|$.
\end{enumerate}

For $x \in X(s)$, $|\Phi_{t, s}(x)| = (t/s)|x| \leq t$, hence $\Phi_{t, s}$ maps $X(s)$ bijectively onto $X(t)$. The identities $\Phi_{t, t}(x) = x$ and $\Phi_{t, s}(\Phi_{s, r}(x)) = \Phi_{t, r}(x)$ hold. For $s \geq 1$, the map $(t, x) \mapsto (t/s)x$ is continuous. 

For $\gamma_i(t) = t x_i$, evaluating the maximum yields $\mathcal{D}_{[a, b]}(\gamma_1, \gamma_2) = b |x_1 - x_2|$. The mapping $(a, b) \mapsto b|x_1 - x_2|$ is continuous for $a \leq b$. The boundary limit yields $\mathcal{D}_{[t, t]}(\gamma_1, \gamma_2) = t|x_1 - x_2| = |\gamma_1(t) - \gamma_2(t)|$.
\end{example}

\begin{example}
Let $V: \mathbb{R}^n \to \mathbb{R}^n$ be a globally Lipschitz continuous vector field. Let $T = \mathbb{R}$ and $X = \mathbb{R}^n$ with the Euclidean metric $d$. Define $(\mathcal{X}, \Phi, \Gamma, \mathcal{D})$ by:
\begin{enumerate}
    \item $X(t) = \mathbb{R}^n$ for all $t \in T$.
    \item $\Phi_{t, s}(x)$ is the value at time $t$ of the unique solution to $\dot{y} = V(y)$ with $y(s) = x$.
    \item $\Gamma$ is the set of all global integral curves of $V$.
    \item For $a < b$, $\mathcal{D}_{[a, b]}(\gamma_1, \gamma_2) = \frac{1}{b - a} \int_{a}^{b} d(\gamma_1(\tau), \gamma_2(\tau)) \, d\tau$. For $a = b = t$, $\mathcal{D}_{[t, t]}(\gamma_1, \gamma_2) = d(\gamma_1(t), \gamma_2(t))$.
\end{enumerate}

Since $V$ is globally Lipschitz on $\mathbb{R}^n$, every initial value problem admits a unique solution defined for all $t \in \mathbb{R}$. Consequently, $\Phi_{t, t} = \mathrm{Id}_{\mathbb{R}^n}$ and $\Phi_{t, s} \circ \Phi_{s, r} = \Phi_{t, r}$, the latter following from the uniqueness of integral curves. Moreover, the flow map $(t, x) \mapsto \Phi_{t, s}(x)$ is continuous.

For trajectories $\gamma_1, \gamma_2 \in \Gamma$, define the distance function:
$$
    f(\tau) = d(\gamma_1(\tau), \gamma_2(\tau)).
$$
Because $\gamma_1, \gamma_2$, and the metric $d$ are continuous, $f: \mathbb{R} \to \mathbb{R}$ is continuous. Hence, for $a < b$, the map:
$$
    (a, b) \longmapsto \frac{1}{b - a} \int_a^b f(\tau) \, d\tau
$$
is continuous on $\{(a, b) \in \mathbb{R}^2 \mid a < b\}$. Furthermore, the continuity of $f$ implies:
$$
    \lim_{\substack{a, b \to t \\ a < b}} \frac{1}{b - a} \int_a^b f(\tau) \, d\tau = f(t) = d(\gamma_1(t), \gamma_2(t)),
$$
so the family of metrics $\mathcal{D}$ satisfies the required continuity conditions. 

Thus, every globally Lipschitz continuous vector field on $\mathbb{R}^n$ canonically determines a temporal metric space.
\end{example}

\section{Topological Structure of the Trajectory Space}

Let $(\mathcal{X}, \Phi, \Gamma, \mathcal{D})$ be a temporal metric space. The family of temporal metrics $\{\mathcal{D}_{[a, b]}\}$ induces a natural convergence structure on $\Gamma$ defined intervalwise as follows.

\begin{definition}
Let $\{\gamma_n\} \subset \Gamma$ be a sequence of trajectories, and let $\gamma \in \Gamma$. The sequence \textit{converges to $\gamma$ on the interval} $[a, b] \subseteq T$ if $\lim_{n \to \infty} \mathcal{D}_{[a, b]}(\gamma_n, \gamma) = 0$. Furthermore, the sequence \textit{converges to $\gamma$ in} $\Gamma$, denoted $\gamma_n \to \gamma$, if it converges to $\gamma$ on every compact interval $[a, b] \subseteq T$.
\end{definition}

\begin{definition}
Let $\{\gamma_n\} \subset \Gamma$ be a sequence of trajectories. The sequence is \textit{Cauchy on the interval} $[a, b] \subseteq T$ if for every $\epsilon > 0$, there exists $N \in \mathbb{N}$ such that $\mathcal{D}_{[a, b]}(\gamma_n, \gamma_m) < \epsilon$ for all $n, m \geq N$. The sequence is \textit{Cauchy in} $\Gamma$ if it is Cauchy on every compact interval $[a, b] \subseteq T$.
\end{definition}

\begin{definition}
\label{def:lsc_tms}
Let $(\mathcal{X}, \Phi, \Gamma, \mathcal{D})$ be a temporal metric space. The space is \textit{lower semi-continuous on the interval} $[a, b] \subseteq T$ if for every $\gamma \in \Gamma$, and every sequence $\{\alpha_m\} \subset \Gamma$ converging pointwise to $\alpha \in \Gamma$ on $[a, b]$, the interval metric satisfies:
$$
    \mathcal{D}_{[a, b]}(\gamma, \alpha) \le \liminf_{m \to \infty} \mathcal{D}_{[a, b]}(\gamma, \alpha_m).
$$
The temporal metric space is \textit{lower semi-continuous} if this condition holds on every compact interval $[a, b] \subseteq T$.
\end{definition}

\begin{definition}
\label{def:strong_tms}
Let $(\mathcal{X}, \Phi, \Gamma, \mathcal{D})$ be a temporal metric space. The space is \textit{strongly continuous on the interval} $[a, b] \subseteq T$ if for every sequence $\{\alpha_m\} \subset \Gamma$ converging pointwise to $\alpha \in \Gamma$ on $[a, b]$, the sequence converges in the interval metric:
$$
    \lim_{m \to \infty} \mathcal{D}_{[a, b]}(\alpha_m, \alpha) = 0.
$$
The temporal metric space is \textit{strongly continuous} if this condition holds on every compact interval $[a, b] \subseteq T$.
\end{definition}

\begin{proposition}
Every strongly continuous temporal metric space is lower semi-continuous.
\end{proposition}

\begin{proof}
Let $[a, b] \subseteq T$ be a closed interval, let $\gamma \in \Gamma$, and let $\{\alpha_m\} \subset \Gamma$ be a sequence converging pointwise to $\alpha \in \Gamma$ on $[a, b]$. By the triangle inequality for the metric $\mathcal{D}_{[a, b]}$:
\[
    \mathcal{D}_{[a, b]}(\gamma, \alpha) \le \mathcal{D}_{[a, b]}(\gamma, \alpha_m) + \mathcal{D}_{[a, b]}(\alpha_m, \alpha).
\]
Since the space is strongly continuous, $\lim_{m \to \infty} \mathcal{D}_{[a, b]}(\alpha_m, \alpha) = 0$. Taking the limit inferior of both sides yields:
\[
    \mathcal{D}_{[a, b]}(\gamma, \alpha) \le \liminf_{m \to \infty} \mathcal{D}_{[a, b]}(\gamma, \alpha_m).
\]
Thus, the space is lower semi-continuous.
\end{proof}

The following example demonstrates that lower semi-continuity is a strictly weaker condition than strong continuity, as pointwise convergence of trajectories does not automatically guarantee convergence in the interval metric.

\begin{example}\label{ex:paris_metric}
Let $(X, d) = (\mathbb{R}^2, \|\cdot\|_2)$ and $T = \mathbb{R}$. Define the temporal metric space $(\mathcal{X}, \Phi, \Gamma, \mathcal{D})$ by:
\begin{enumerate}
    \item $X(t) = \mathbb{R}^2$ for all $t \in T$.
    \item $\Phi_{t, s}(x) = e^{t-s}x$ for all $x \in X(s)$ and $s, t \in T$.
    \item $\Gamma = \{ \gamma_p \in C(T, \mathbb{R}^2) \mid \gamma_p(t) = e^t p, \; p \in \mathbb{R}^2 \}$.
    \item Define the equivalence relation $x \sim y$ if there exists $\lambda > 0$ such that $x = \lambda y$. Let $d_P(x, y)$ be the French Railway (Paris) metric:
    $$
        d_P(x, y) = 
        \begin{cases} 
        \|x - y\|_2 & \text{if } x \sim y \\
        \|x\|_2 + \|y\|_2 & \text{if } x \not\sim y.
        \end{cases}
    $$
    It is a classical result that $d_P$ defines a valid metric on $\mathbb{R}^2$. For any compact interval $[a, b] \subseteq T$ with length $\ell = b - a$, and for any $\alpha, \beta \in \Gamma$, set:
    $$
        \mathcal{D}_{[a, b]}(\alpha, \beta) = \|\alpha(b) - \beta(b)\|_2 + \ell d_P(\alpha(b), \beta(b)).
    $$
\end{enumerate}

Since $e^{t-s} > 0$, the evolution map preserves equivalence classes ($x \sim y \iff \Phi_{t,s}(x) \sim \Phi_{t,s}(y)$). For any trajectories $\alpha, \beta \in \Gamma$ determined by their initial states $p = \alpha(0)$ and $q = \beta(0)$, we have $\alpha(b) = e^b p$ and $\beta(b) = e^b q$. The interval metric evaluates explicitly to:
$$
    \mathcal{D}_{[a, b]}(\alpha, \beta) = e^b \|p - q\|_2 + (b - a) e^b d_P(p, q).
$$
Because this expression is a simple composition of continuous exponential and linear functions in the variables $a$ and $b$, the mapping $(a, b) \mapsto \mathcal{D}_{[a, b]}(\alpha, \beta)$ is continuous for $a \leq b$. On degenerate intervals where $a = b$, the length vanishes, reducing the metric to $\mathcal{D}_{[t, t]}(\alpha, \beta) = e^t \|p - q\|_2 = \|\alpha(t) - \beta(t)\|_2$.

To demonstrate that this space is not strongly continuous, consider the compact interval $[0, 1]$. Let $\alpha_m \in \Gamma$ be the sequence determined by $\alpha_m(0) = (1, 1/m)$, which converges pointwise on $[0, 1]$ to the trajectory $\alpha \in \Gamma$ determined by $\alpha(0) = (1, 0)$. Since $\alpha_m(0) \not\sim \alpha(0)$ for all $m$, we have $\alpha_m(1) \not\sim \alpha(1)$. Evaluating the limit:
$$
\begin{aligned}
    \lim_{m \to \infty} \mathcal{D}_{[0, 1]}(\alpha_m, \alpha) &= \lim_{m \to \infty} \Big( \|(e, \frac{e}{m}) - (e, 0)\|_2 + 1 \cdot \big( \|(e, \frac{e}{m})\|_2 + \|(e, 0)\|_2 \big) \Big) \\
    &= \lim_{m \to \infty} \left( \frac{e}{m} + \sqrt{e^2 + \frac{e^2}{m^2}} + e \right) = 2e \neq 0.
\end{aligned}
$$
Thus, strong continuity fails.

To verify lower semi-continuity, let $\gamma \in \Gamma$ and let $\alpha_m \to \alpha$ pointwise on an arbitrary interval $[a, b]$. Define the spatial boundary points $Z = \gamma(b)$, $X_m = \alpha_m(b)$, and $X = \alpha(b)$. Pointwise convergence implies $X_m \to X$ in the Euclidean norm. 
\begin{itemize}
    \item If $X = Z$, then $\mathcal{D}_{[a, b]}(\gamma, \alpha) = 0 \leq \liminf_{m \to \infty} \mathcal{D}_{[a, b]}(\gamma, \alpha_m)$ holds trivially.
    \item If $X \neq Z$ and $X \not\sim Z$, the ray $R_Z = \{ \lambda Z \mid \lambda \ge 0 \}$ is closed. Because $X \notin R_Z$, $X_m \notin R_Z$ for sufficiently large $m$, forcing $X_m \not\sim Z$. Taking the limit inferior:
    \begin{multline*}
    \liminf_{m \to \infty} \Big( \|Z - X_m\|_2 + \ell (\|Z\|_2 + \|X_m\|_2) \Big) \\
    = \|Z - X\|_2 + \ell (\|Z\|_2 + \|X\|_2) = \mathcal{D}_{[a, b]}(\gamma, \alpha).
    \end{multline*}
    \item If $X \neq Z$ and $X \sim Z$, the bound $d_P(u, v) \geq \|u - v\|_2$ yields:
    \begin{multline*}
    \liminf_{m \to \infty} \mathcal{D}_{[a, b]}(\gamma, \alpha_m) \geq \liminf_{m \to \infty} \Big( \|Z - X_m\|_2 + \ell \|Z - X_m\|_2 \Big) \\
    = (1+\ell)\|Z - X\|_2 = \|Z - X\|_2 + \ell d_P(Z, X) = \mathcal{D}_{[a, b]}(\gamma, \alpha).
    \end{multline*}
\end{itemize}
In all cases, $\mathcal{D}_{[a, b]}(\gamma, \alpha) \leq \liminf_{m \to \infty} \mathcal{D}_{[a, b]}(\gamma, \alpha_m)$. The temporal metric space is lower semi-continuous.
\end{example}

With the convergence structure of trajectories fully established, we now formalize the notion of completeness directly within the trajectory space.

\begin{definition}
A temporal metric space $(\mathcal{X}, \Phi, \Gamma, \mathcal{D})$ is complete if its trajectory space $\Gamma$ is complete; that is, if every Cauchy sequence in $\Gamma$ converges to a unique limit trajectory $\gamma \in \Gamma$.
\end{definition}

We now establish sufficient conditions under which this property is inherited from the base space.

\begin{theorem}
\label{thm:inheritance}
Let $(\mathcal{X}, \Phi, \Gamma, \mathcal{D})$ be a lower semi-continuous temporal metric space over a base space $(X, d)$. If $(X, d)$ is complete and there exists $t_0 \in T$ such that $X(t_0)$ is closed in $X$, then $(\mathcal{X}, \Phi, \Gamma, \mathcal{D})$ is complete.
\end{theorem}

\begin{proof}
Let $\{\gamma_n\} \subset \Gamma$ be a Cauchy sequence. For $t_0 \in T$, the degenerate interval metric yields $\mathcal{D}_{[t_0, t_0]}(\gamma_n, \gamma_m) = d(\gamma_n(t_0), \gamma_m(t_0))$. Since $\{\gamma_n\}$ is Cauchy in $\Gamma$, $\{\gamma_n(t_0)\}$ is Cauchy in $(X, d)$. 

Because $(X, d)$ is complete, $\{\gamma_n(t_0)\}$ converges to some $x^* \in X$. Since $X(t_0)$ is closed in $X$, $x^* \in X(t_0)$. 

Define $\gamma: T \to X$ by:
$$
    \gamma(t) = \Phi_{t, t_0}(x^*)
$$

We verify $\gamma \in \Gamma$. For any $s \in T$, we have:
$$
    \gamma(t) = \Phi_{t, t_0}(x^*) = \Phi_{t, s}(\Phi_{s, t_0}(x^*)) = \Phi_{t, s}(\gamma(s)).
$$
Since the map $t \mapsto \Phi_{t, t_0}(x^*)$ is continuous, $\gamma \in C(T, X)$. Thus, $\gamma \in \Gamma$.

For any $t \in T$, $\gamma_n(t) = \Phi_{t, t_0}(\gamma_n(t_0))$. By the continuity of $\Phi_{t, t_0}: X(t_0) \to X(t)$, we have:
$$
    \lim_{n \to \infty} \gamma_n(t) = \lim_{n \to \infty} \Phi_{t, t_0}(\gamma_n(t_0)) = \Phi_{t, t_0}(x^*) = \gamma(t).
$$

To show $\gamma_n \to \gamma$ in $\Gamma$, let $[a, b] \subseteq T$. Since $\{\gamma_n\}$ is Cauchy, for any $\epsilon > 0$, there exists $N \in \mathbb{N}$ such that $\mathcal{D}_{[a, b]}(\gamma_n, \gamma_m) < \epsilon$ for all $n, m \geq N$. 

Fix $n \geq N$ and let $m \to \infty$. Since $\gamma_m \to \gamma$ pointwise on $[a, b]$, by lower semi-continuity we have:
$$
    \mathcal{D}_{[a, b]}(\gamma_n, \gamma) \leq \liminf_{m \to \infty} \mathcal{D}_{[a, b]}(\gamma_n, \gamma_m) \leq \epsilon.
$$
Thus, $\lim_{n \to \infty} \mathcal{D}_{[a, b]}(\gamma_n, \gamma) = 0$. 

To show uniqueness, suppose $\gamma_n$ converges to $\eta$ in $\Gamma$. For any $t \in T$,
$$
    d(\gamma(t), \eta(t)) = \mathcal{D}_{[t, t]}(\gamma, \eta) \leq \mathcal{D}_{[t, t]}(\gamma, \gamma_n) + \mathcal{D}_{[t, t]}(\gamma_n, \eta).
$$
As $n \to \infty$, both terms on the right converge to $0$, yielding $d(\gamma(t), \eta(t)) = 0$. Hence, $\gamma(t) = \eta(t)$ for all $t \in T$, and the limit $\gamma$ is unique.
\end{proof}

\begin{example}
\label{ex:necessity_lsc}
Let the base space be $X = \mathbb{R}^2$ equipped with the Euclidean metric $d(u,v) = \|u-v\|_2$, and let $T = \mathbb{R}$. We construct a temporal metric space over $X$ that fails to be lower semi-continuous and, in this specific construction, is incomplete. This demonstrates that the lower semi-continuity hypothesis in Theorem \ref{thm:inheritance} cannot be dropped, as the completeness of the base space $(X, d)$ alone is insufficient to guarantee the completeness of the trajectory space.

Define the temporal metric space $(\mathcal{X}, \Phi, \Gamma, \mathcal{D})$ by:
\begin{enumerate}
    \item $X(t) = \mathbb{R}^2$ for all $t \in T$.
    \item $\Phi_{t, s}(v) = R_{t-s}v$, where $R_\theta$ denotes the counterclockwise rotation of $\mathbb{R}^2$ by angle $\theta$ about the origin.
    \item $\Gamma = \{\gamma \in C(T, \mathbb{R}^2) \mid \gamma(t) = R_t p \text{ for some } p \in \mathbb{R}^2\}$.
  \item Define the function $\mathcal{E}: \mathbb{R}^2 \to \mathbb{R}$ by $\mathcal{E}(x, y) = x^2 / (x^2 + y^2)$ for $(x, y) \neq (0,0)$, and $\mathcal{E}(0, 0) = 0$. For any compact interval $[a, b] \subseteq T$ with length $\ell = b - a$, and for any trajectories $\alpha, \beta \in \Gamma$, define the temporal metric:
    $$
        \mathcal{D}_{[a, b]}(\alpha, \beta) = \|\alpha(b) - \beta(b)\|_2 + \ell \left| \sup_{t \in [a, b]} \mathcal{E}(\alpha(t)) - \sup_{t \in [a, b]} \mathcal{E}(\beta(t)) \right|.
    $$
\end{enumerate}

Non-negativity, symmetry, and the triangle inequality for $\mathcal{D}_{[a, b]}$ follow directly from the properties of the Euclidean norm and the supremum. If $\mathcal{D}_{[a, b]}(\alpha, \beta) = 0$, then $\|\alpha(b) - \beta(b)\|_2 = 0$, which implies $\alpha(b) = \beta(b)$. By Lemma \ref{lem:global_identity}, trajectories in $\Gamma$ are uniquely determined by their state at any single instant; thus, it follows that $\alpha = \beta$. Furthermore, for any fixed trajectory, the map $t \mapsto \mathcal{E}(\alpha(t))$ is continuous. This guarantees that the mapping $(a, b) \mapsto \mathcal{D}_{[a, b]}(\alpha, \beta)$ is continuous. Finally, on degenerate intervals where $a = b$, the length $\ell$ vanishes, and the temporal metric reduces to the spatial Euclidean distance $\mathcal{D}_{[t, t]}(\alpha, \beta) = \|\alpha(t) - \beta(t)\|_2$. Consequently, all axioms of a temporal metric space are satisfied.

To demonstrate the failure of lower semi-continuity, consider the compact interval $[0, \pi]$, yielding $\ell = \pi$. Define a sequence of trajectories $\alpha_m \in \Gamma$ uniquely determined by the state $\alpha_m(0) = (1/m, 0)$. As $m \to \infty$, $\alpha_m$ converges pointwise on $[0, \pi]$ to the zero trajectory $\alpha_0(t) \equiv (0,0)$. Let $\gamma \in \Gamma$ be the reference trajectory determined by $\gamma(0) = (1,0)$. 

Because both the sequence trajectories $\alpha_m$ and the reference trajectory $\gamma$ intersect the $x$-axis at non-zero points during the interval $[0, \pi]$, they achieve the maximal value of $\mathcal{E}$, yielding $\sup_{t \in [0, \pi]} \mathcal{E}(\alpha_m(t)) = \sup_{t \in [0, \pi]} \mathcal{E}(\gamma(t)) = 1$ for all $m \ge 1$. In contrast, the pointwise limit $\alpha_0$ remains entirely at the origin, where the function is explicitly defined as $\mathcal{E}(0,0) = 0$, resulting in $\sup_{t \in [0, \pi]} \mathcal{E}(\alpha_0(t)) = 0$. Evaluating the limit inferior of the sequence yields:

$$\liminf_{m \to \infty} \mathcal{D}_{[0, \pi]}(\gamma, \alpha_m) = \liminf_{m \to \infty} \left( \left\Vert{} (-1, 0) - (-1/m, 0) \right\Vert{}_2 + \pi \vert{}1 - 1\vert{} \right) = 1.$$
However, direct evaluation gives $\mathcal{D}_{[0, \pi]}(\gamma, \alpha_0) = 1 + \pi$. Since $1 + \pi > 1$, lower semi-continuity fails.

Furthermore, $\Gamma$ is incomplete. For $m, n \ge 1$, we have:
$$
    \mathcal{D}_{[0, \pi]}(\alpha_m, \alpha_n) = \|(-1/m, 0) - (-1/n, 0)\|_2 + \pi |1 - 1| = \left| \frac{1}{m} - \frac{1}{n} \right|.
$$
Since this distance vanishes as $m, n \to \infty$, the sequence $\{\alpha_m\}_{m=1}^\infty$ is Cauchy in $\Gamma$. Any limit of this sequence in $\Gamma$ must coincide with its pointwise limit $\alpha_0$. However, the temporal distance to $\alpha_0$ is:
$$
    \mathcal{D}_{[0, \pi]}(\alpha_m, \alpha_0) = \|(-1/m, 0) - (0, 0)\|_2 + \pi |1 - 0| = \frac{1}{m} + \pi.
$$
Taking the limit as $m \to \infty$ yields $\pi > 0$. Thus, the Cauchy sequence fails to converge in $\Gamma$, and the temporal metric space is incomplete.
\end{example}

\subsection{Sequential Compactness in the Trajectory Space}

To characterize sequentially compact subsets of the trajectory space $\Gamma$, we first define trajectory cross-sections and trajectory bundles at a fixed time.

\begin{definition}
Let $\mathcal{A} \subseteq \Gamma$ be a family of trajectories. For any fixed time $t \in T$, the cross-section of $\mathcal{A}$ at time $t$ is the subset of the base space defined by:
$$
    \mathcal{A}(t) = \{\alpha(t) \in X(t) \mid \alpha \in \mathcal{A}\} \subseteq X.
$$
Conversely, for any subset $C \subseteq X(t)$, the trajectory bundle passing through $C$ at time $t$ is the maximal set of trajectories intersecting $C$:
$$
    \Gamma(C, t) = \{\alpha \in \Gamma \mid \alpha(t) \in C\} \subseteq \Gamma.
$$
\end{definition}

\begin{definition}
A subset $\mathcal{A} \subseteq \Gamma$ is relatively sequentially compact if every sequence in $\mathcal{A}$ contains a subsequence converging to a trajectory in $\Gamma$. The subset $\mathcal{A}$ is sequentially compact if every sequence in $\mathcal{A}$ contains a subsequence converging to a trajectory in $\mathcal{A}$.
\end{definition}

By Lemma \ref{lem:global_identity}, a trajectory in $\Gamma$ is uniquely determined by its value at any fixed time $t_0 \in T$. Under the assumption of strong continuity, this property characterizes sequential compactness in $\Gamma$.

\begin{theorem}
\label{thm:characterization_compactness}
Let $(\mathcal{X}, \Phi, \Gamma, \mathcal{D})$ be a strongly continuous temporal metric space over $(X, d)$, and suppose there exists $t_0 \in T$ such that $X(t_0)$ is closed in $X$.
\begin{enumerate}
    \item[(i)] A subset $\mathcal{A} \subseteq \Gamma$ is relatively sequentially compact in $\Gamma$ if and only if its cross-section $\mathcal{A}(t_0)$ is relatively sequentially compact in $X$.
    \item[(ii)] A subset $\mathcal{A} \subseteq \Gamma$ is sequentially compact in $\Gamma$ if and only if its cross-section $\mathcal{A}(t_0)$ is sequentially compact in $X$.
\end{enumerate}
\end{theorem}

\begin{proof}
(i) Suppose $\mathcal{A}(t_0)$ is relatively sequentially compact. Let $\{\alpha_n\}$ be a sequence in $\mathcal{A}$. The sequence $\{ \alpha_n(t_0) \}$ in $\mathcal{A}(t_0)$ contains a subsequence $\{ \alpha_{n_k}(t_0) \}$ converging to some $x^* \in X$. Since $X(t_0)$ is closed, $x^* \in X(t_0)$. Define $\alpha(t) = \Phi_{t, t_0}(x^*)$. Because $t \mapsto \Phi_{t, t_0}(x^*)$ is continuous, $\alpha \in \Gamma$. For any $t \in T$, the continuity of $\Phi_{t, t_0}$ implies $\lim_{k \to \infty} \alpha_{n_k}(t) = \Phi_{t, t_0}(x^*) = \alpha(t)$. By strong continuity, $\lim_{k \to \infty} \mathcal{D}_{[a, b]}(\alpha_{n_k}, \alpha) = 0$ for every compact interval $[a, b] \subseteq T$. Thus, $\alpha_{n_k} \to \alpha$ in $\Gamma$, and $\mathcal{A}$ is relatively sequentially compact.

Conversely, suppose $\mathcal{A}$ is relatively sequentially compact. Let $\{x_n\}$ be a sequence in $\mathcal{A}(t_0)$. Choose $\alpha_n \in \mathcal{A}$ such that $\alpha_n(t_0) = x_n$. There exists a subsequence $\{\alpha_{n_k}\}$ converging in $\Gamma$ to some $\alpha^* \in \Gamma$. Evaluating over the degenerate interval $[t_0, t_0]$ yields:
$$
    d(x_{n_k}, \alpha^*(t_0)) = \mathcal{D}_{[t_0, t_0]}(\alpha_{n_k}, \alpha^*) \to 0.
$$
Because $\alpha^* \in \Gamma$, we have $\alpha^*(t_0) \in X(t_0)$. Thus, $\{x_n\}$ has a convergent subsequence in $X(t_0)$, proving $\mathcal{A}(t_0)$ is relatively sequentially compact.

(ii) Suppose $\mathcal{A}(t_0)$ is sequentially compact. Let $\{\beta_n\}$ be a sequence in $\mathcal{A}$. The sequence $\{ \beta_n(t_0) \}$ in $\mathcal{A}(t_0)$ contains a subsequence $\{ \beta_{n_k}(t_0) \}$ converging to some $y^* \in \mathcal{A}(t_0)$. Define $\beta(t) = \Phi_{t, t_0}(y^*)$. Since $y^* \in X(t_0)$, it follows that $\beta \in \Gamma$. For any $t \in T$, the continuity of $\Phi_{t, t_0}$ implies $\lim_{k \to \infty} \beta_{n_k}(t) = \Phi_{t, t_0}(y^*) = \beta(t)$. By strong continuity, $\beta_{n_k} \to \beta$ in $\Gamma$. Because $y^* \in \mathcal{A}(t_0)$, there exists some $\eta \in \mathcal{A}$ such that $\eta(t_0) = y^* = \beta(t_0)$. By Lemma \ref{lem:global_identity}, $\beta = \eta$, ensuring $\beta \in \mathcal{A}$. Thus, $\mathcal{A}$ is sequentially compact.

Conversely, suppose $\mathcal{A}$ is sequentially compact. Let $\{x_n\}$ be a sequence in $\mathcal{A}(t_0)$, and choose $\alpha_n \in \mathcal{A}$ such that $\alpha_n(t_0) = x_n$. There exists a subsequence $\{\alpha_{n_k}\}$ converging in $\Gamma$ to some $\alpha^* \in \mathcal{A}$. Evaluating over the degenerate interval $[t_0, t_0]$ yields:
$$
    d(x_{n_k}, \alpha^*(t_0)) = \mathcal{D}_{[t_0, t_0]}(\alpha_{n_k}, \alpha^*) \to 0.
$$
Since $\alpha^* \in \mathcal{A}$, we have $\alpha^*(t_0) \in \mathcal{A}(t_0)$. Thus, $\{x_n\}$ has a convergent subsequence in $\mathcal{A}(t_0)$, proving $\mathcal{A}(t_0)$ is sequentially compact.
\end{proof}

\begin{example}
\label{ex:necessity_strong_continuity}
We construct a temporal metric space that satisfies lower semi-continuity but fails to be strongly continuous, in which a trajectory bundle over a sequentially compact cross-section fails to be sequentially compact.

Let $(X, d) = (\mathbb{R}, |\cdot|)$ and $T = \mathbb{R}$. Define the evolution family by $\Phi_{t, s}(x) = x$, yielding the trajectory space $\Gamma = \{\gamma_x \in C(T, \mathbb{R}) \mid \gamma_x(t) = x \text{ for all } t \in T \text{ and some } x \in \mathbb{R}\}$. Let $\delta(x, y)$ be the discrete metric on $\mathbb{R}$. For any compact interval $[a, b] \subseteq T$ with length $\ell = b - a$, define:
$$
    \mathcal{D}_{[a, b]}(\gamma_x, \gamma_y) = |x - y| + \ell \delta(x, y).
$$

Let $\gamma_z \in \Gamma$, and let a sequence $\gamma_{x_m}$ converge pointwise on $[a, b]$ to $\gamma_x$. If $x = z$, then $\mathcal{D}_{[a, b]}(\gamma_z, \gamma_x) = 0 \le \liminf_{m \to \infty} \mathcal{D}_{[a, b]}(\gamma_z, \gamma_{x_m})$. If $x \neq z$, then for sufficiently large $m$, $x_m \neq z$, yielding $\delta(z, x_m) = 1$. Thus:
$$
    \liminf_{m \to \infty} \mathcal{D}_{[a, b]}(\gamma_z, \gamma_{x_m}) = \liminf_{m \to \infty} (|z - x_m| + \ell) = |z - x| + \ell = \mathcal{D}_{[a, b]}(\gamma_z, \gamma_x).
$$
The space is lower semi-continuous.

Consider the interval $[0, 1]$ of length $\ell = 1$. The sequence $\gamma_{1/m}$ converges pointwise to $\gamma_0$. Evaluating the interval metric gives:
$$
    \lim_{m \to \infty} \mathcal{D}_{[0, 1]}(\gamma_{1/m}, \gamma_0) = \lim_{m \to \infty} \left( \left| \frac{1}{m} - 0 \right| + 1 \right) = 1 \neq 0.
$$
Thus, strong continuity fails.

Let $C = [0, 1]$, which is sequentially compact in $X$. Consider the trajectory bundle $\Gamma(C, 0)$ passing through $C$ at $t_0 = 0$. The sequence of trajectories $\{\gamma_{1/m}\}$ is contained in $\Gamma(C, 0)$. For any compact interval $[a, b]$ with length $\ell > 0$, and for any $m \neq n$:
$$
    \mathcal{D}_{[a, b]}(\gamma_{1/m}, \gamma_{1/n}) = \left| \frac{1}{m} - \frac{1}{n} \right| + \ell > \ell.
$$
Because the trajectories remain uniformly separated by at least $\ell > 0$ in the interval metric, no subsequence can be Cauchy. Consequently, no subsequence converges in $\Gamma$, proving the bundle $\Gamma(C, 0)$ is not relatively sequentially compact.
\end{example}

\subsection{Equivariance and Lifted Mappings}

Before constructing specific mappings between trajectory spaces, we formalize the general notion of continuity between them.

\begin{definition}
Let $(\mathcal{X}, \Phi, \Gamma, \mathcal{D})$ and $(\mathcal{Y}, \Psi, \Lambda, \mathcal{E})$ be temporal metric spaces. A mapping $F: \Gamma \to \Lambda$ is \textit{temporally continuous} if it preserves the sequential convergence of trajectories; that is, for any sequence $\{\alpha_n\}_{n=1}^\infty \subset \Gamma$ and trajectory $\alpha \in \Gamma$, the convergence $\alpha_n \to \alpha$ in $\Gamma$ implies $F(\alpha_n) \to F(\alpha)$ in $\Lambda$.
\end{definition}

We now lift mappings between base spaces to the trajectory spaces.

\begin{definition}
Let $(\mathcal{X}, \Phi, \Gamma, \mathcal{D})$ and $(\mathcal{Y}, \Psi, \Lambda, \mathcal{E})$ be temporal metric spaces over base spaces $X$ and $Y$, respectively, with a common time domain $T$. A mapping $h: X \to Y$ is dynamically equivariant if for all $t \in T$, $h(X(t)) \subseteq Y(t)$, and for all $s, t \in T$ and $x \in X(s)$:
$$
    h(\Phi_{t, s}(x)) = \Psi_{t, s}(h(x)).
$$
\end{definition}

\begin{definition}
Let $(\mathcal{X}, \Phi, \Gamma, \mathcal{D})$ and $(\mathcal{Y}, \Psi, \Lambda, \mathcal{E})$ be temporal metric spaces over base spaces $X$ and $Y$, respectively, with a common time domain $T$. Let $h: X \to Y$ be a continuous, dynamically equivariant mapping. The temporal map $\tilde{h}: \Gamma \to \Lambda$ induced by $h$ is defined for each $\alpha \in \Gamma$ by:
$$
    \tilde{h}(\alpha)(t) = h(\alpha(t)) \quad \text{for all } t \in T.
$$
\end{definition}

To verify $\tilde{h}$ is well-defined, let $\alpha \in \Gamma$. Since $\alpha$ and $h$ are continuous, the composition $\tilde{h}(\alpha)$ is in $C(T, Y)$. By dynamical equivariance, for all $s, t \in T$:
$$
    \tilde{h}(\alpha)(t) = h(\alpha(t)) = h(\Phi_{t, s}(\alpha(s))) = \Psi_{t, s}(h(\alpha(s))) = \Psi_{t, s}(\tilde{h}(\alpha)(s)).
$$
Thus, $\tilde{h}(\alpha) \in \Lambda$.

\begin{theorem}
\label{thm:lifting_continuity}
Let $(\mathcal{X}, \Phi, \Gamma, \mathcal{D})$ and $(\mathcal{Y}, \Psi, \Lambda, \mathcal{E})$ be temporal metric spaces over base spaces $X$ and $Y$, respectively, with a common time domain $T$. Let $h: X \to Y$ be a continuous, dynamically equivariant mapping. If the target space is strongly continuous, then the temporal map $\tilde{h}: \Gamma \to \Lambda$ induced by $h$ is temporally continuous.
\end{theorem}

\begin{proof}
Let $\{\alpha_n\}$ be a sequence converging to $\alpha$ in $\Gamma$. Evaluating over the degenerate interval $[t, t]$ yields:
$$
    d(\alpha_n(t), \alpha(t)) = \mathcal{D}_{[t, t]}(\alpha_n, \alpha) \to 0.
$$
Thus, $\alpha_n(t) \to \alpha(t)$ for all $t \in T$. By the continuity of $h$, $h(\alpha_n(t)) \to h(\alpha(t))$ in $Y$, which by definition means $\tilde{h}(\alpha_n)(t) \to \tilde{h}(\alpha)(t)$ for all $t \in T$. Because the target space is strongly continuous, this pointwise convergence implies $\tilde{h}(\alpha_n) \to \tilde{h}(\alpha)$ in $\Lambda$. Thus, $\tilde{h}$ is temporally continuous.
\end{proof}

\begin{corollary}
Let $(\mathcal{X}, \Phi, \Gamma, \mathcal{D})$ and $(\mathcal{Y}, \Psi, \Lambda, \mathcal{E})$ be strongly continuous temporal metric spaces over base spaces $X$ and $Y$, respectively, with a common time domain $T$. Let $h: X \to Y$ be a dynamically equivariant homeomorphism. Then the temporal map $\tilde{h}: \Gamma \to \Lambda$ induced by $h$ is a temporal homeomorphism.
\end{corollary}

\begin{proof}
Because $h$ is bijective, the pointwise definition of $\tilde{h}$ ensures it is bijective. Specifically, $\tilde{h}(\alpha) = \tilde{h}(\beta)$ implies $h(\alpha(t)) = h(\beta(t))$, yielding $\alpha(t) = \beta(t)$ for all $t \in T$. Surjectivity follows symmetrically, meaning the inverse map $\tilde{h}^{-1}: \Lambda \to \Gamma$ is the temporal map induced by $h^{-1}$.

The inverse $h^{-1}: Y \to X$ is dynamically equivariant, as $h(\Phi_{t, s}(h^{-1}(y))) = \Psi_{t, s}(y)$ yields $\Phi_{t, s}(h^{-1}(y)) = h^{-1}(\Psi_{t, s}(y))$. 

Because $h$ and $h^{-1}$ are both continuous and dynamically equivariant, applying Theorem \ref{thm:lifting_continuity} to both mappings establishes that $\tilde{h}$ and $\tilde{h}^{-1}$ are temporally continuous. Thus, $\tilde{h}$ is a temporal homeomorphism.
\end{proof}

\begin{example}
Let $T = \mathbb{R}$, and let the base spaces $(X, d)$ and $(Y, \rho)$ both be $\mathbb{R}^2$ equipped with the usual Euclidean metric.

Define a temporal metric space $(\mathcal{X}, \Phi, \Gamma, \mathcal{D})$ over $X$ by:
\begin{enumerate}
    \item $X(t) = \mathbb{R}^2$ for all $t \in T$.
    \item $\Phi_{t, s}(x) = e^{t-s}x$ for all $s, t \in T$.
    \item $\Gamma = \{\gamma \in C(T, \mathbb{R}^2) \mid \gamma(t) = e^t p \text{ for some } p \in \mathbb{R}^2\}$.
    \item $\mathcal{D}_{[a, b]}(\alpha, \beta) = \sup_{t \in [a, b]} d(\alpha(t), \beta(t))$.
\end{enumerate}
Uniform convergence on compact intervals preserves continuous limits; therefore, $\mathcal{X}$ is strongly continuous.

Define a temporal metric space $(\mathcal{Y}, \Psi, \Lambda, \mathcal{E})$ over $Y$ with identical fibers $Y(t) = \mathbb{R}^2$, evolution family $\Psi_{t, s}(y) = e^{t-s}y$, and trajectory space $\Lambda = \Gamma$. Let the temporal metric $\mathcal{E}_{[a, b]}$ be defined as in Example \ref{ex:paris_metric}. As established, $\mathcal{Y}$ is not strongly continuous.

Define the identity mapping $h: X \to Y$ by $h(x) = x$. The mapping $h$ is continuous on the base spaces and is dynamically equivariant:
$$
    h(\Phi_{t, s}(x)) = e^{t-s}x = \Psi_{t, s}(h(x)).
$$

Consider the sequence of trajectories $\alpha_m \in \Gamma$ determined by $\alpha_m(0) = (1, 1/m)$, and the reference trajectory $\alpha \in \Gamma$ determined by $\alpha(0) = (1, 0)$. Let $[a, b] \subseteq T$ be an arbitrary compact interval, we have

\begin{align*}
    \lim_{m \to \infty} \mathcal{D}_{[a, b]}(\alpha_m, \alpha) &= \lim_{m \to \infty} \sup_{t \in [a, b]} d(\alpha_m(t), \alpha(t)) \\ 
    &= \lim_{m \to \infty} \sup_{t \in [a, b]} e^t d\left(\left(1, \frac{1}{m}\right), (1, 0)\right) = \lim_{m \to \infty} \frac{e^b}{m} = 0.
\end{align*}

Because this holds for every compact interval $[a, b]$, $\alpha_m \to \alpha$ in $\Gamma$.

The induced temporal map $\tilde{h}: \Gamma \to \Lambda$ evaluates identically as $\tilde{h}(\gamma)(t) = \gamma(t)$. Consequently, $\tilde{h}(\alpha_m) = \alpha_m$ and $\tilde{h}(\alpha) = \alpha$. As in Example \ref{ex:paris_metric}, for the compact interval $[0, 1]$, we have
$$
    \lim_{m \to \infty} \mathcal{E}_{[0, 1]}(\tilde{h}(\alpha_m), \tilde{h}(\alpha)) = 2e \neq 0.
$$
The sequence converges in $\mathcal{X}$ but fails to converge in $\mathcal{Y}$. Thus, the lifted map $\tilde{h}$ is discontinuous, proving that strong continuity of the target space is required in Theorem \ref{thm:lifting_continuity} for the lifted map to be continuous.
\end{example}

\section{The Temporal Hyperspace}

We now construct a topology over sets of trajectories. Throughout this work, hyperspaces are formulated using sequential compactness to align with the ambient trajectory topology.

\begin{definition}
Let $(\mathcal{X}, \Phi, \Gamma, \mathcal{D})$ be a temporal metric space over a base space $X$ with time domain $T$. The temporal hyperspace, denoted $\mathcal{K}(\Gamma)$, is the collection of all non-empty, sequentially compact subsets of $\Gamma$. For any $t \in T$, the hyperspace of the cross-section $X(t)$, denoted $\mathcal{K}(X(t))$, is the collection of all non-empty, sequentially compact subsets of $X(t)$.
\end{definition}

By Theorem \ref{thm:characterization_compactness}, in strongly continuous spaces, elements of $\mathcal{K}(\Gamma)$ coincide exactly with trajectory bundles emanating from sequentially compact cross-sections. This establishes a strict identification between the temporal hyperspace and the hyperspace of the cross-section:

\begin{theorem}
\label{thm:hyperspace_reduction}
Let $(\mathcal{X}, \Phi, \Gamma, \mathcal{D})$ be a strongly continuous temporal metric space over a base space $X$ with time domain $T$. If there exists $t_0 \in T$ such that $X(t_0)$ is closed in $X$, then the map $\Theta: \mathcal{K}(\Gamma) \to \mathcal{K}(X(t_0))$ defined by $\Theta(\mathcal{A}) = \mathcal{A}(t_0)$ is a bijection.
\end{theorem}

\begin{proof}
By Theorem \ref{thm:characterization_compactness}, for any $\mathcal{A} \in \mathcal{K}(\Gamma)$, its cross-section $\mathcal{A}(t_0)$ is sequentially compact, ensuring $\Theta$ is well-defined. To see that $\Theta$ is surjective, let $C \in \mathcal{K}(X(t_0))$. Theorem \ref{thm:characterization_compactness} guarantees that the trajectory bundle $\Gamma(C, t_0) \in \mathcal{K}(\Gamma)$ and its cross-section is precisely $C$; hence $\Theta(\Gamma(C, t_0)) = C$. To see that $\Theta$ is injective, suppose $\Theta(\mathcal{A}) = \Theta(\mathcal{B})$ for some $\mathcal{A}, \mathcal{B} \in \mathcal{K}(\Gamma)$, meaning $\mathcal{A}(t_0) = \mathcal{B}(t_0)$. By Theorem \ref{thm:characterization_compactness}, every compact trajectory set is strictly determined by its cross-section, so $\mathcal{A} = \Gamma(\mathcal{A}(t_0), t_0) = \Gamma(\mathcal{B}(t_0), t_0) = \mathcal{B}$. Thus, $\Theta$ is a bijection.
\end{proof}

\begin{definition}
Let $\mathcal{A}, \mathcal{B} \in \mathcal{K}(\Gamma)$. For any compact interval $[a, b] \subseteq T$, the directed temporal Hausdorff distance from $\mathcal{A}$ to $\mathcal{B}$ is:
$$
    \vec{\mathcal{H}}_{[a, b]}(\mathcal{A}, \mathcal{B}) = \sup_{\alpha \in \mathcal{A}} \inf_{\beta \in \mathcal{B}} \mathcal{D}_{[a, b]}(\alpha, \beta).
$$
The temporal Hausdorff metric over $[a, b]$ is the symmetrized maximum:
$$
    \mathcal{H}_{[a, b]}(\mathcal{A}, \mathcal{B}) = \max \left\{ \vec{\mathcal{H}}_{[a, b]}(\mathcal{A}, \mathcal{B}), \vec{\mathcal{H}}_{[a, b]}(\mathcal{B}, \mathcal{A}) \right\}.
$$
\end{definition}

Because $\mathcal{D}_{[a, b]}$ is continuous on $\Gamma \times \Gamma$ and the trajectory sets $\mathcal{A}$ and $\mathcal{B}$ are sequentially compact, the infimum and supremum are achieved. The family of metrics $\mathcal{H} = \{\mathcal{H}_{[a, b]}\}$ generates the topology of the temporal hyperspace.

\begin{proposition}
For any compact interval $[a, b] \subseteq T$, the temporal Hausdorff metric $\mathcal{H}_{[a, b]}$ is a true metric on $\mathcal{K}(\Gamma)$.
\end{proposition}

\begin{proof}
Non-negativity and symmetry follow immediately from the definition of the symmetrized maximum and the properties of $\mathcal{D}_{[a, b]}$. To verify the triangle inequality, let $\mathcal{A}, \mathcal{B}, \mathcal{C} \in \mathcal{K}(\Gamma)$. For any $\alpha \in \mathcal{A}$, $\beta \in \mathcal{B}$, and $\gamma \in \mathcal{C}$, the triangle inequality of the underlying temporal metric yields:
$$
    \mathcal{D}_{[a, b]}(\alpha, \gamma) \le \mathcal{D}_{[a, b]}(\alpha, \beta) + \mathcal{D}_{[a, b]}(\beta, \gamma).
$$
Taking the infimum over all $\gamma \in \mathcal{C}$, we obtain:
$$
    \inf_{\gamma \in \mathcal{C}} \mathcal{D}_{[a, b]}(\alpha, \gamma) \le \mathcal{D}_{[a, b]}(\alpha, \beta) + \inf_{\gamma \in \mathcal{C}} \mathcal{D}_{[a, b]}(\beta, \gamma) \le \mathcal{D}_{[a, b]}(\alpha, \beta) + \vec{\mathcal{H}}_{[a, b]}(\mathcal{B}, \mathcal{C}).
$$
Taking the infimum over $\beta \in \mathcal{B}$ on the right side:
$$
    \inf_{\gamma \in \mathcal{C}} \mathcal{D}_{[a, b]}(\alpha, \gamma) \le \inf_{\beta \in \mathcal{B}} \mathcal{D}_{[a, b]}(\alpha, \beta) + \vec{\mathcal{H}}_{[a, b]}(\mathcal{B}, \mathcal{C}).
$$
Taking the supremum over all $\alpha \in \mathcal{A}$ gives the directed triangle inequality:
$$
    \vec{\mathcal{H}}_{[a, b]}(\mathcal{A}, \mathcal{C}) \le \vec{\mathcal{H}}_{[a, b]}(\mathcal{A}, \mathcal{B}) + \vec{\mathcal{H}}_{[a, b]}(\mathcal{B}, \mathcal{C}).
$$
By symmetric reasoning, the reverse directed distance satisfies:
$$
    \vec{\mathcal{H}}_{[a, b]}(\mathcal{C}, \mathcal{A}) \le \vec{\mathcal{H}}_{[a, b]}(\mathcal{C}, \mathcal{B}) + \vec{\mathcal{H}}_{[a, b]}(\mathcal{B}, \mathcal{A}).
$$
By definition of the symmetric maximum, $\vec{\mathcal{H}}_{[a, b]}(\mathcal{M}, \mathcal{N}) \le \mathcal{H}_{[a, b]}(\mathcal{M}, \mathcal{N})$ for any $\mathcal{M}, \mathcal{N} \in \mathcal{K}(\Gamma)$. Applying this to the right sides of both inequalities yields:
\begin{align*}
    \vec{\mathcal{H}}_{[a, b]}(\mathcal{A}, \mathcal{C}) &\le \mathcal{H}_{[a, b]}(\mathcal{A}, \mathcal{B}) + \mathcal{H}_{[a, b]}(\mathcal{B}, \mathcal{C}), \\
    \vec{\mathcal{H}}_{[a, b]}(\mathcal{C}, \mathcal{A}) &\le \mathcal{H}_{[a, b]}(\mathcal{B}, \mathcal{C}) + \mathcal{H}_{[a, b]}(\mathcal{A}, \mathcal{B}).
\end{align*}
Because both $\vec{\mathcal{H}}_{[a, b]}(\mathcal{A}, \mathcal{C})$ and $\vec{\mathcal{H}}_{[a, b]}(\mathcal{C}, \mathcal{A})$ are less than or equal to the exact same sum, their maximum must also satisfy this inequality, yielding the triangle inequality:
$$
    \mathcal{H}_{[a, b]}(\mathcal{A}, \mathcal{C}) \le \mathcal{H}_{[a, b]}(\mathcal{A}, \mathcal{B}) + \mathcal{H}_{[a, b]}(\mathcal{B}, \mathcal{C}).
$$

Suppose $\mathcal{H}_{[a, b]}(\mathcal{A}, \mathcal{B}) = 0$, which implies both directed distances are zero. Because $\vec{\mathcal{H}}_{[a, b]}(\mathcal{A}, \mathcal{B}) = 0$ and the sets are sequentially compact, for any $\alpha \in \mathcal{A}$ there exists $\beta \in \mathcal{B}$ such that $\mathcal{D}_{[a, b]}(\alpha, \beta) = 0$. By the definition of a temporal metric space, $\mathcal{D}_{[a, b]}$ is a true metric on the trajectory space $\Gamma$, which strictly requires $\alpha = \beta$. Thus, $\alpha \in \mathcal{B}$, establishing $\mathcal{A} \subseteq \mathcal{B}$. By symmetric reasoning, $\vec{\mathcal{H}}_{[a, b]}(\mathcal{B}, \mathcal{A}) = 0$ yields $\mathcal{B} \subseteq \mathcal{A}$. Thus, $\mathcal{A} = \mathcal{B}$, confirming $\mathcal{H}_{[a, b]}$ is a metric.
\end{proof}

\begin{definition}
A sequence $\{\mathcal{A}_n\}$ in $\mathcal{K}(\Gamma)$ converges to a limit set $\mathcal{A}^* \in \mathcal{K}(\Gamma)$ if, for every compact interval $J \subseteq T$:
$$
\lim_{n \to \infty} \mathcal{H}_J(\mathcal{A}_n, \mathcal{A}^*) = 0.$$
\end{definition}

For a deterministic dynamical fractal to emerge as a unique, stable limit of an iterative process, the hyperspace hosting the process must be complete. The following theorem establishes that completeness seamlessly lifts from the trajectory space to the hyperspace.

\begin{theorem}
\label{thm:hyperspace_completeness}
Let $(\mathcal{X}, \Phi, \Gamma, \mathcal{D})$ be a lower semi-continuous temporal metric space over a complete base space $(X, d)$, and suppose there exists a reference time $t_0 \in T$ such that the fiber $X(t_0)$ is closed in $X$. Then the associated temporal hyperspace $(\mathcal{K}(\Gamma), \mathcal{H})$ consisting of all non-empty sequentially compact subsets of $\Gamma$ is complete under the induced temporal Hausdorff metric.
\end{theorem}

\begin{proof}
By Theorem \ref{thm:inheritance}, the hypotheses guarantee that the underlying temporal metric space is complete.
Let $\{\mathcal{A}_n\}_{n=1}^\infty$ be a Cauchy sequence in $(\mathcal{K}(\Gamma), \mathcal{H})$. To prove completeness, we construct a limit set $\mathcal{A}^* \in \mathcal{K}(\Gamma)$ and demonstrate that $\lim_{n \to \infty} \mathcal{H}_I(\mathcal{A}_n, \mathcal{A}^*) = 0$ for every compact interval $I \subseteq T$. Without loss of generality, let $t_0 = 0$.

Step 1: The Pointwise Limit. 

On the degenerate interval $[0, 0]$, the temporal metric evaluates to the base metric: $\mathcal{D}_{[0,0]}(\alpha, \beta) = d(\alpha(0), \beta(0))$. Because $\{\mathcal{A}_n\}$ is Cauchy in the hyperspace, the sequence of cross-sections $\{\mathcal{A}_n(0)\}$ forms a Cauchy sequence of compact sets in the complete metric space $(X, d)$. By classical Hausdorff completeness, $\{\mathcal{A}_n(0)\}$ converges to a unique compact limit set $A_0 \subseteq X(0)$.

We define the global candidate limit set $\mathcal{A}^*$ as the bundle of all trajectories originating from $A_0$:
\[
    \mathcal{A}^* = \{ \alpha \in \Gamma \mid \alpha(0) \in A_0 \}
\]
Because $A_0$ is non-empty, $\mathcal{A}^*$ is non-empty. 

Step 2: Forward Directed Distance Bound. ($\vec{\mathcal{H}}_I(\mathcal{A}_n, \mathcal{A}^*) \to 0$) 

Let $I \subseteq T$ be an arbitrary compact interval. Fix $\epsilon > 0$. Since $\{\mathcal{A}_n\}$ is Cauchy on $I$, there exists $N \in \mathbb{N}$ such that for all $n, m \ge N$, $\mathcal{H}_I(\mathcal{A}_n, \mathcal{A}_m) < \epsilon$.

Fix $n \ge N$ and select an arbitrary trajectory $\alpha_n \in \mathcal{A}_n$. For each $m \ge N$, the bound $\mathcal{H}_I(\mathcal{A}_n, \mathcal{A}_m) < \epsilon$ guarantees the existence of a trajectory $\gamma_m \in \mathcal{A}_m$ satisfying $\mathcal{D}_I(\alpha_n, \gamma_m) < \epsilon$. Let $x_m = \gamma_m(0) \in \mathcal{A}_m(0)$. Because $\mathcal{A}_m(0) \to A_0$ in the Hausdorff metric of the base space, the sequence $\{x_m\}_{m \ge N}$ contains a convergent subsequence $x_{m_k} \to x^* \in A_0$. Let $\alpha^* \in \mathcal{A}^*$ be the unique trajectory originating from $x^*$. Consequently, $\gamma_{m_k}$ converges pointwise to $\alpha^*$ for all $t \in T$.

Applying lower semi-continuity (LSC) to the reference trajectory $\alpha_n$ and the pointwise-converging sequence $\gamma_{m_k}$:
\[
    \mathcal{D}_I(\alpha_n, \alpha^*) \le \liminf_{k \to \infty} \mathcal{D}_I(\alpha_n, \gamma_{m_k}) \le \epsilon
\]
Taking the supremum over all $\alpha_n \in \mathcal{A}_n$ yields $\vec{\mathcal{H}}_I(\mathcal{A}_n, \mathcal{A}^*) \le \epsilon$.

Step 3: Reverse Directed Distance Bound. ($\vec{\mathcal{H}}_I(\mathcal{A}^*, \mathcal{A}_n) \to 0$) 

Fix $\alpha^* \in \mathcal{A}^*$. By definition, $\alpha^*(0) \in A_0$. Because $\mathcal{A}_m(0) \to A_0$, there exists a sequence of points $y_m \in \mathcal{A}_m(0)$ converging to $\alpha^*(0)$. The corresponding trajectories $\eta_m \in \mathcal{A}_m$ satisfying $\eta_m(0) = y_m$ converge pointwise to $\alpha^*$.

For our fixed $n \ge N$ and any $m \ge N$, the bound $\mathcal{H}_I(\mathcal{A}_n, \mathcal{A}_m) < \epsilon$ guarantees the existence of a trajectory $\beta_{n,m} \in \mathcal{A}_n$ such that $\mathcal{D}_I(\beta_{n,m}, \eta_m) < \epsilon$. Because $\mathcal{A}_n$ is sequentially compact in $\Gamma$, the sequence $\{\beta_{n,m}\}_{m=N}^\infty$ possesses a convergent subsequence $\beta_{n, m_k} \to \hat{\alpha}_n \in \mathcal{A}_n$ under the interval metric $\mathcal{D}_I$.

Applying LSC to the reference trajectory $\hat{\alpha}_n$ and the pointwise-converging sequence $\eta_{m_k}$:
\[
    \mathcal{D}_I(\hat{\alpha}_n, \alpha^*) \le \liminf_{k \to \infty} \mathcal{D}_I(\hat{\alpha}_n, \eta_{m_k})
\]
By the triangle inequality:
\begin{align*}
    \liminf_{k \to \infty} \mathcal{D}_I(\hat{\alpha}_n, \eta_{m_k}) &\le \liminf_{k \to \infty} \left( \mathcal{D}_I(\hat{\alpha}_n, \beta_{n, m_k}) + \mathcal{D}_I(\beta_{n, m_k}, \eta_{m_k}) \right) \\
    &\le 0 + \epsilon = \epsilon
\end{align*}
Thus, $\mathcal{D}_I(\hat{\alpha}_n, \alpha^*) \le \epsilon$. Taking the supremum over all $\alpha^* \in \mathcal{A}^*$ yields $\vec{\mathcal{H}}_I(\mathcal{A}^*, \mathcal{A}_n) \le \epsilon$. Combining Steps 2 and 3 establishes $\mathcal{H}_I(\mathcal{A}_n, \mathcal{A}^*) \to 0$ for every compact interval $I \subseteq T$.

Step 4: Sequential Compactness of the Limit Set. 

To complete the proof, we must verify that $\mathcal{A}^* \in \mathcal{K}(\Gamma)$, which requires $\mathcal{A}^*$ to be sequentially compact in $\Gamma$. 
Let $\{\alpha_k\}_{k=1}^\infty \subseteq \mathcal{A}^*$ be an arbitrary sequence. At $t=0$, the points $\alpha_k(0)$ reside in the compact set $A_0$. Thus, there exists a single subsequence $\alpha_{k_j}(0)$ converging to some $x^* \in A_0$. Let $\alpha^* \in \mathcal{A}^*$ be the trajectory satisfying $\alpha^*(0) = x^*$. Because the evolution map $\Phi_{t, 0}$ is continuous, $\alpha_{k_j}(t) \to \alpha^*(t)$ for all $t \in T$.

We claim this specific subsequence $\alpha_{k_j}$ converges to $\alpha^*$ under every compact interval metric $\mathcal{D}_I$. Let $I \subseteq T$ be an arbitrary compact interval. Assume for contradiction that $\alpha_{k_j}$ does not converge to $\alpha^*$ under $\mathcal{D}_I$. Then there exists $\epsilon_0 > 0$ and a further subsequence $\{\eta_i\} \subseteq \{\alpha_{k_j}\}$ such that $\mathcal{D}_I(\eta_i, \alpha^*) \ge \epsilon_0$ for all $i$.

Because $\mathcal{H}_I(\mathcal{A}_n, \mathcal{A}^*) \to 0$ and each $\mathcal{A}_n$ is sequentially compact (hence totally bounded), the limit set $\mathcal{A}^*$ inherits total boundedness under $\mathcal{D}_I$. Because $\{\eta_i\}$ resides in a totally bounded space under $\mathcal{D}_I$, it must possess a Cauchy subsequence $\{\mu_m\}$ under $\mathcal{D}_I$. Because $\{\mu_m\}$ is a subsequence of $\{\alpha_{k_j}\}$, it preserves the pointwise convergence $\mu_m \to \alpha^*$. 

Applying LSC to the Cauchy sequence $\{\mu_m\}$: for any $\delta > 0$, there exists $M$ such that for all $p, q \ge M$, $\mathcal{D}_I(\mu_p, \mu_q) < \delta$. Fixing $p \ge M$:
\[
    \mathcal{D}_I(\mu_p, \alpha^*) \le \liminf_{q \to \infty} \mathcal{D}_I(\mu_p, \mu_q) \le \delta
\]
As $\delta$ was arbitrary, $\lim_{p \to \infty} \mathcal{D}_I(\mu_p, \alpha^*) = 0$. Thus, $\{\mu_m\}$ strictly converges to $\alpha^*$ under $\mathcal{D}_I$. However, $\{\mu_m\}$ is a subsequence of $\{\eta_i\}$, which was strictly bounded away from $\alpha^*$ by $\epsilon_0$, forming a contradiction.

Therefore, the original subsequence $\alpha_{k_j}$ converges to $\alpha^*$ under $\mathcal{D}_I$. Because $I$ was arbitrary and the choice of the subsequence $\alpha_{k_j}$ was completely independent of $I$, $\alpha_{k_j}$ converges to $\alpha^*$ on all compact intervals simultaneously. Thus, $\alpha_{k_j} \to \alpha^*$ in $\Gamma$, proving $\mathcal{A}^*$ is sequentially compact. Consequently, $\mathcal{A}^* \in \mathcal{K}(\Gamma)$, and the temporal hyperspace is complete.
\end{proof}

\section{Temporal Iterated Function Systems}

In a temporal metric space $(\mathcal{X}, \Phi, \Gamma, \mathcal{D})$, the evolution family $\Phi_{t, s}$ directly connects the fibers across time. Therefore, an iterated function system defined on a single reference fiber $X(t_0)$ induces a corresponding IFS on every fiber $X(t)$ via topological conjugation. We now formalize these flow-induced dynamics and prove that the evolution family strictly propagates the attractor across time, ensuring its topological stability even when ambient contractivity fails.

\begin{definition}
Let $(\mathcal{X}, \Phi, \Gamma, \mathcal{D})$ be a temporal metric space. For a fixed time $t_0 \in T$, given a function $h_{t_0}: X(t_0) \to X(t_0)$, its \textit{flow-induced function} $h_t: X(t) \to X(t)$ at any time $t \in T$ is defined by the conjugation:
$$
    h_t = \Phi_{t, t_0} \circ h_{t_0} \circ \Phi_{t_0, t}.
$$
The collection $\{h_t\}_{t \in T}$ is the \textit{induced dynamically equivariant family} of $h_{t_0}$.
\end{definition}

\begin{definition}
Let $(\mathcal{X}, \Phi, \Gamma, \mathcal{D})$ be a temporal metric space. For a fixed time $t_0 \in T$, let $\mathcal{I}_{t_0} = \{h_{1, t_0}, \dots, h_{N, t_0}\}$ be an iterated function system on the fiber $X(t_0)$. The \textit{flow-induced iterated function system} at time $t \in T$ on $X(t)$ is the collection $\mathcal{I}_t = \{h_{1, t}, \dots, h_{N, t}\}$, where $h_{i, t} = \Phi_{t, t_0} \circ h_{i, t_0} \circ \Phi_{t_0, t}$. The family $\mathcal{I} = \{\mathcal{I}_t\}_{t \in T}$ is called the \textit{temporal iterated function system} (TIFS) generated by the system $\mathcal{I}_{t_0}$.
\end{definition}

\begin{definition}
Let $(\mathcal{X}, \Phi, \Gamma, \mathcal{D})$ be a temporal metric space, and let $\mathcal{I}$ be a temporal iterated function system generated by a reference IFS $\mathcal{I}_{t_0}$ on the fiber $X(t_0)$. If $A_{t_0} \subseteq X(t_0)$ is an attractor of $\mathcal{I}_{t_0}$, the \textit{temporal attractor} of $\mathcal{I}$ associated with $A_{t_0}$ is defined as the trajectory bundle $\mathcal{A}^*$ passing through $A_{t_0}$ at time $t_0$:
$$
    \mathcal{A}^* = \{\alpha \in \Gamma \mid \alpha(t_0) \in A_{t_0}\}.
$$
\end{definition}

\begin{theorem}\label{thm:topological_stability}
Let $(\mathcal{X}, \Phi, \Gamma, \mathcal{D})$ be a temporal metric space on the base space $(X, d)$. For a fixed time $t_0 \in T$, let $\mathcal{I}_{t_0}$ be an IFS on the fiber $X(t_0)$ with a unique, compact attractor $A_{t_0}$, and let $\mathcal{A}^*$ be its associated temporal attractor. Then, for any $t \in T$, the cross-section $\mathcal{A}^*(t)$ is the unique, compact attractor of the flow-induced IFS $\mathcal{I}_t$ on the fiber $X(t)$, and it satisfies:
$$
    \mathcal{A}^*(t) = \Phi_{t, t_0}(A_{t_0}).
$$
Furthermore, this attractor exists even if the flow-induced functions $h_{i, t}$ fail to be contractions in the metric $d$.
\end{theorem}

\begin{proof}
By definition, the cross-section of the temporal attractor at time $t$ is $\mathcal{A}^*(t) = \{\alpha(t) \mid \alpha \in \mathcal{A}^*\}$. Because every trajectory $\alpha \in \mathcal{A}^*$ satisfies $\alpha(t) = \Phi_{t, t_0}(\alpha(t_0))$ and $\alpha(t_0) \in A_{t_0}$, it follows immediately that $\mathcal{A}^*(t) = \Phi_{t, t_0}(A_{t_0})$.

Because $A_{t_0}$ is a compact subset of $X(t_0)$ and the flow $\Phi_{t, t_0}$ is continuous, its image $\mathcal{A}^*(t)$ is a compact subset of $X(t)$.

We first establish that $\mathcal{A}^*(t)$ is invariant under the Hutchinson operator $\mathcal{W}_t$ associated with the flow-induced IFS $\mathcal{I}_t$. By definition, this operator evaluates as:
\begin{align*}
    \mathcal{W}_t(S) &= \bigcup_{i=1}^N h_{i, t}(S) = \bigcup_{i=1}^N \big(\Phi_{t, t_0} \circ h_{i, t_0} \circ \Phi_{t_0, t}\big)(S) \\
    &= \Phi_{t, t_0}\left( \bigcup_{i=1}^N h_{i, t_0}\big(\Phi_{t_0, t}(S)\big) \right) = \big(\Phi_{t, t_0} \circ \mathcal{W}_{t_0} \circ \Phi_{t_0, t}\big)(S).
\end{align*}
Evaluating this conjugated operator on the target set $\mathcal{A}^*(t)$ yields:
$$
    \mathcal{W}_t(\mathcal{A}^*(t)) = \big(\Phi_{t, t_0} \circ \mathcal{W}_{t_0} \circ \Phi_{t_0, t}\big) \big(\Phi_{t, t_0}(A_{t_0})\big).
$$
Since the flow operators are mutual inverses ($\Phi_{t_0, t} \circ \Phi_{t, t_0} = \mathrm{Id}$), this simplifies to:
$$
    \mathcal{W}_t(\mathcal{A}^*(t)) = \Phi_{t, t_0}\big(\mathcal{W}_{t_0}(A_{t_0})\big).
$$
Because $A_{t_0}$ is the unique attractor of the reference IFS, it is invariant under $\mathcal{W}_{t_0}$, meaning $\mathcal{W}_{t_0}(A_{t_0}) = A_{t_0}$. Substituting this identity yields:
$$
    \mathcal{W}_t(\mathcal{A}^*(t)) = \Phi_{t, t_0}(A_{t_0}) = \mathcal{A}^*(t).
$$
Thus, $\mathcal{A}^*(t)$ is invariant.

Finally, we determine dynamical stability. Let $C' \subset X(t)$ be an arbitrary compact set, and define $C = \Phi_{t_0, t}(C')$. Because $\Phi_{t_0, t}$ is continuous, $C$ is a compact subset of $X(t_0)$. Because $A_{t_0}$ is the unique attractor of $\mathcal{I}_{t_0}$, the sequence of iterates $\mathcal{W}_{t_0}^m(C)$ converges in the Hausdorff metric to $A_{t_0}$ as $m \to \infty$.

Evaluating the $m$-th iteration of $\mathcal{W}_t$ on $C'$ gives:
$$
    \mathcal{W}_t^m(C') = (\Phi_{t, t_0} \circ \mathcal{W}_{t_0} \circ \Phi_{t_0, t})^m(C') = \Phi_{t, t_0} \big( \mathcal{W}_{t_0}^m(C) \big).
$$
A continuous function between metric spaces induces a continuous function between their respective hyperspaces of compact sets equipped with the Hausdorff metric. Because the flow operator $\Phi_{t, t_0}$ is continuous, we may pass the limit directly through it:
$$
    \lim_{m \to \infty} \mathcal{W}_t^m(C') = \Phi_{t, t_0}\left( \lim_{m \to \infty} \mathcal{W}_{t_0}^m(C) \right) = \Phi_{t, t_0}(A_{t_0}) = \mathcal{A}^*(t).
$$
Because $\mathcal{A}^*(t)$ is compact, invariant, and the sequence of iterates $\mathcal{W}_t^m(C')$ converges to $\mathcal{A}^*(t)$ in the Hausdorff metric for every compact set $C' \subset X(t)$, it is the unique attractor of the flow-induced IFS $\mathcal{I}_t$.
\end{proof}

\begin{remark}
If the temporal metric space is lower semi-continuous and the base space is complete, the proof of Theorem \ref{thm:hyperspace_completeness} guarantees that the temporal attractor $\mathcal{A}^*$ is sequentially compact in $\Gamma$. Consequently, $\mathcal{A}^* \in \mathcal{K}(\Gamma)$. Thus, the temporal attractor is properly characterized as a single compact element of the temporal hyperspace, with the flow-induced attractors $\mathcal{A}^*(t)$ arising strictly as its cross-sections.
\end{remark}

\begin{theorem}\label{thm:isotopic_conjugacy}
Let $\mathcal{I}_0 = \{f_1, \dots, f_N\}$ and $\mathcal{I}^* = \{g_1, \dots, g_N\}$ be iterated function systems on a complete metric space $(X, d)$. Assume that each $g_i$ is a strict contraction on $(X, d)$, and that $\mathcal{I}_0$ is topologically conjugate to $\mathcal{I}^*$ via a homeomorphism $H: X \to X$ satisfying $f_i = H^{-1} \circ g_i \circ H$ for all $i \in \{1, \dots, N\}$.

If $H$ is isotopic to the identity map on $X$, then there exists a temporal metric space $(\mathcal{X}, \Phi, \Gamma, \mathcal{D})$ over $(X, d)$ with time domain $T = [0, 1]$ such that the flow-induced IFS at time $t=1$, denoted $\mathcal{I}_1$, consists entirely of strict contractions on $(X, d)$.
\end{theorem}

\begin{proof}
Define the fibers as $X(t) = X$ for all $t \in T$. By hypothesis, $H$ is isotopic to the identity. Thus, there exists a continuous map $F: T \times X \to X$, denoted $F(t, x) = H_t(x)$, such that $H_0 = \mathrm{Id}_X$, $H_1 = H$, and for each $t \in T$, the map $H_t: X \to X$ is a homeomorphism.

Define the evolution family $\Phi$ for all $s, t \in T$ by $\Phi_{t, s} = H_t \circ H_s^{-1}$. This definition satisfies the identity and composition axioms:
\begin{align*}
    \Phi_{t, t} &= H_t \circ H_t^{-1} = \mathrm{Id}_X, \\
    \Phi_{t, s} \circ \Phi_{s, r} &= H_t \circ H_s^{-1} \circ H_s \circ H_r^{-1} = H_t \circ H_r^{-1} = \Phi_{t, r}.
\end{align*}
For any fixed $s \in T$, the inverse map $x \mapsto H_s^{-1}(x)$ is continuous on $X$. Because the isotopy $(t, y) \mapsto H_t(y)$ is continuous on $T \times X$, their composition $(t, x) \mapsto H_t(H_s^{-1}(x)) = \Phi_{t, s}(x)$ is continuous. 

The evolution family $\Phi$ uniquely determines the trajectory space $\Gamma$, which, when equipped with any valid interval metric $\mathcal{D}$, formally completes the temporal metric space. By definition and the topological conjugacy, the flow-induced mappings $h_i \in \mathcal{I}_1$ evaluate to:
\begin{align*}
    h_i &= \Phi_{1, 0} \circ f_i \circ \Phi_{0, 1} = H_1 \circ H_0^{-1} \circ f_i \circ H_0 \circ H_1^{-1} \\
    &= H \circ f_i \circ H^{-1} = g_i.
\end{align*}

Since $h_i = g_i$ and each $g_i$ is a strict contraction, the system $\mathcal{I}_1$ consists entirely of strict contractions on $(X, d)$.
\end{proof}

\begin{example}
We demonstrate that if the conjugacy $H$ is not isotopic to the identity, the construction of a temporal metric space satisfying $\Phi_{1,0} = H$ as described in Theorem \ref{thm:isotopic_conjugacy} may fail.

Let the base space be $\mathbb{R}$ equipped with the standard Euclidean metric $d(x,y) = |x - y|$. Define an IFS $\mathcal{I}^* = \{g_1, g_2\}$ by:
\begin{align*}
    g_1(x) &= \frac{1}{2}x + 2, \\
    g_2(x) &= \frac{1}{2}x - 2.
\end{align*}
Both $g_1$ and $g_2$ are strict contractions with Lipschitz constant $1/2$.

Define the orientation-reversing homeomorphism $H: \mathbb{R} \to \mathbb{R}$ by $H(x) = -x^3$, with inverse $H^{-1}(y) = -\sqrt[3]{y}$. Via the topological conjugacy $f_i = H^{-1} \circ g_i \circ H$, the initial IFS $\mathcal{I}_0 = \{f_1, f_2\}$ evaluates to:
\begin{align*}
    f_1(x) &= -\sqrt[3]{-\frac{1}{2}x^3 + 2}, \\
    f_2(x) &= -\sqrt[3]{-\frac{1}{2}x^3 - 2}.
\end{align*}
Because the derivatives of $f_1$ and $f_2$ are unbounded near $x = \sqrt[3]{4}$ and $x = -\sqrt[3]{4}$ respectively, the mappings fail to be contractions on $(\mathbb{R}, d)$.

Suppose, for contradiction, that a temporal metric space $(\mathcal{X}, \Phi, \Gamma, \mathcal{D})$ over $\mathbb{R}$ with time domain $T = [0, 1]$ exists such that the forward flow equates to the conjugacy, $\Phi_{1,0} = H$. By definition, the map $t \mapsto \Phi_{t, 0}(x)$ must be continuous, and for each $t \in T$, the map $\Phi_{t, 0}: \mathbb{R} \to \mathbb{R}$ must be a homeomorphism.

Define the distance function $D: [0, 1] \to \mathbb{R}$ by:
$$
    D(t) = \Phi_{t, 0}(1) - \Phi_{t, 0}(-1).
$$
The continuity of the flow ensures $D(t)$ is continuous. Evaluating the boundary times yields:
\begin{align*}
    D(0) &= \Phi_{0, 0}(1) - \Phi_{0, 0}(-1) = \mathrm{Id}_{\mathbb{R}}(1) - \mathrm{Id}_{\mathbb{R}}(-1) = 2, \\
    D(1) &= \Phi_{1, 0}(1) - \Phi_{1, 0}(-1) = H(1) - H(-1) = -1 - 1 = -2.
\end{align*}
Since $D(0) > 0$ and $D(1) < 0$, the Intermediate Value Theorem guarantees the existence of a time $t^* \in (0, 1)$ such that $D(t^*) = 0$. 

The condition $D(t^*) = 0$ implies $\Phi_{t^*, 0}(1) = \Phi_{t^*, 0}(-1)$. Thus, $\Phi_{t^*, 0}$ is not injective, contradicting the requirement that it be a homeomorphism. Therefore, such an evolution family $\Phi$ cannot exist.
\end{example}

We conclude with two explicit examples demonstrating Theorem \ref{thm:topological_stability}.

\begin{example}\label{ex:transient_contractivity}
Let the base space be $\mathbb{R}^2$ with the usual Euclidean metric $d$. Define a temporal metric space with time domain $T = [0, \infty)$ and fibers $X(t) = \mathbb{R}^2$ for all $t \in T$. Let the evolution family be:
$$
    \Phi_{t, s}(x, y) = \big(e^{t-s}x, \, y\big).
$$

Let $\mathcal{I}_0 = \{f_0, g_0\}$ be an iterated function system on the fiber $X(0)$, where:
\begin{align*}
    f_0(x, y) &= \left( \frac{y}{4}, \, 2x \right), \\
    g_0(x, y) &= \left( -\frac{x}{2}, \, -\frac{y}{2} + 1 \right).
\end{align*}
The mapping $g_0$ is a contraction with Lipschitz constant $1/2$. However, $f_0$ expands the $x$-coordinates by a factor of $2$. Thus, $\mathcal{I}_0$ fails to be a classical IFS consisting of contractions on $(\mathbb{R}^2, d)$.

The flow-induced iterated function system $\mathcal{I}_t = \{f_t, g_t\}$ on the fiber $X(t)$ evaluates to:
\begin{align*}
    f_t(x, y) &= \big(\Phi_{t, 0} \circ f_0 \circ \Phi_{0, t}\big)(x, y) = \left( \frac{e^t}{4}y, \, 2e^{-t}x \right), \\
    g_t(x, y) &= \big(\Phi_{t, 0} \circ g_0 \circ \Phi_{0, t}\big)(x, y) = \left( -\frac{x}{2}, \, -\frac{y}{2} + 1 \right).
\end{align*}
The Lipschitz constant of $g_t$ remains $1/2$ for all $t \in T$. For $f_t$, the contractivity factor evaluates to:
$$
    L(f_t) = \max\left( \frac{e^t}{4}, \, \frac{2}{e^t} \right).
$$
For $f_t$ to be a contraction, $L(f_t) < 1$ requires:
$$
    \frac{e^t}{4} < 1 \implies t < \ln 4 \quad \text{and} \quad \frac{2}{e^t} < 1 \implies t > \ln 2.
$$
Thus, the mapping $f_t$ is a contraction if and only if $t \in (\ln 2, \ln 4)$. For $t \le \ln 2$ and $t \ge \ln 4$, contractivity fails.

By Theorem \ref{thm:topological_stability}, the corresponding system $\mathcal{I}_t$ of the temporal iterated function system possesses a unique, compact attractor $\mathcal{A}^*(t) = \Phi_{t, 0}(A_0)$ for all $t \in T$. Figure \ref{fig:transient_contractivity} illustrates instances of $\mathcal{A}^*(t)$ at $t=0$, $1$, and $2$.

\begin{figure}[htbp]
    \centering
    \includegraphics[width=\textwidth]{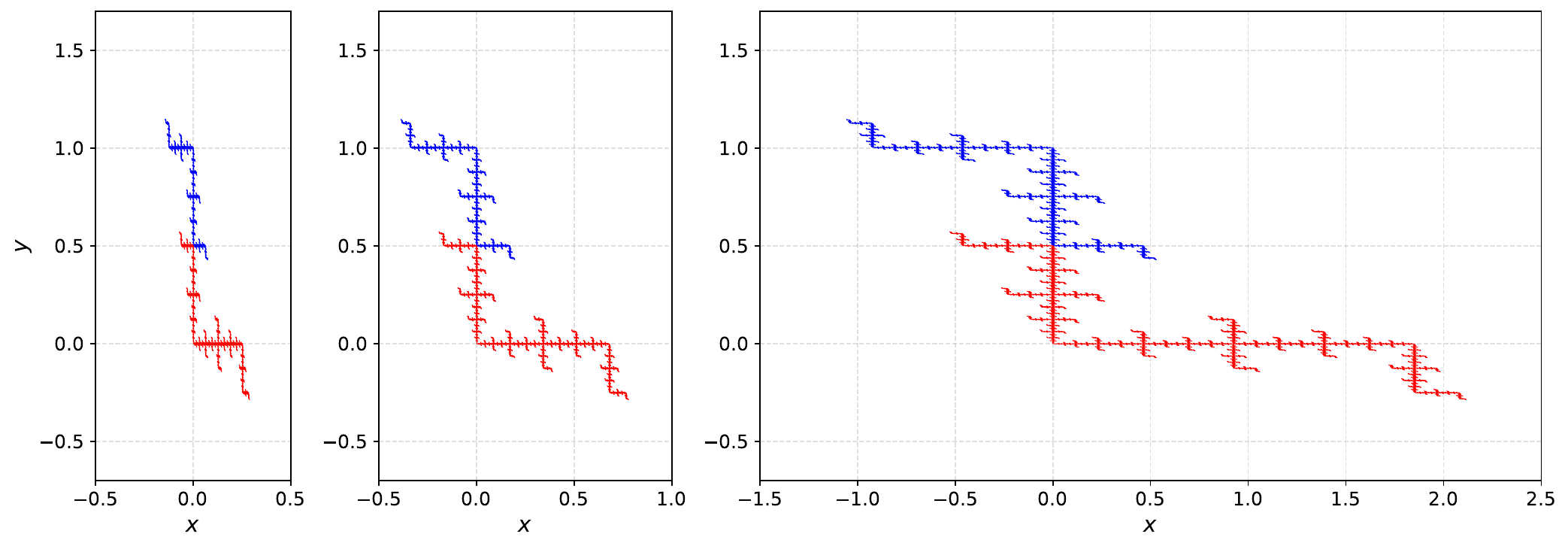}
    \caption{Evolution of the attractor associated with Example \ref{ex:transient_contractivity}, shown from left to right at times $t=0$, $1$, and $2$. The red and blue portions denote the images of the attractor under the mappings $f_t$ and $g_t$, respectively. The system fails to be a classical IFS at $t=0.0$ ($L(f_0) = 2.0$), satisfies contractivity at $t=1.0$ ($L(f_1) \approx 0.73$), and fails at $t=2.0$ ($L(f_2) \approx 1.84$). The attractor persists despite the failure of classical contraction conditions.}
    \label{fig:transient_contractivity}
\end{figure}
\end{example}

\begin{example}\label{ex:conformal_flow}
Let the base space be $X = \mathbb{C}$ equipped with the Euclidean metric $d(z, w) = |z - w|$, and the time domain be $T = [0, \infty)$. Let the evolution family be the fractional linear transformation:
$$
    \Phi_{t, s}(z) = \frac{z}{1 - 0.8i(\sin t - \sin s)z}.
$$
The inverse flow to $t_0 = 0$ is $\Phi_{0, t}(z) = z / (1 + 0.8i\sin(t)z)$.

To satisfy the axioms of a temporal metric space, the evolution family must act as a homeomorphism between fibers. Define the reference fiber as the compact disk $X(0) = \{z \in \mathbb{C} \mid |z| \le 1.1\}$, and define subsequent fibers dynamically as $X(t) = \Phi_{t, 0}(X(0))$ for all $t \in T$. For any $s, t \in T$, the map $\Phi_{t, s}: X(s) \to X(t)$ is well-defined provided its pole does not lie in $X(s)$. Equivalently, the pole pulled back to the reference fiber via the inverse flow $\Phi_{0, s}$ must not lie in $X(0)$. By the composition axiom $\Phi_{t, s} \circ \Phi_{s, 0} = \Phi_{t, 0}$, this pre-image $z_0 \in \mathbb{C}$ is simply the pole of $\Phi_{t, 0}$, which occurs when:
$$
    1 - 0.8i\sin(t)z_0 = 0 \implies z_0 = \frac{1}{0.8i\sin t} = \frac{-i}{0.8\sin t}.
$$
The modulus of this pole evaluates to:
$$
    |z_0| = \frac{1}{0.8|\sin t|} \ge \frac{1}{0.8} = 1.25.
$$
Because the reference fiber $X(0)$ is strictly bounded by $|z| \le 1.1$, the pole $z_0$ never enters $X(0)$ for any $t \in T$. Consequently, the singularities remain outside every spatial fiber $X(s)$, ensuring that each evolution map $\Phi_{t,s}$ is a well-defined homeomorphism on its domain.

Let $\mathcal{I}_0 = \{f_0, g_0, h_0\}$ be an IFS on the fiber $X(0)$, where:
\begin{align*}
    f_0(z) &= 0.5z + 0.4i, \\
    g_0(z) &= 0.5z - 0.4 - 0.3i, \\
    h_0(z) &= 0.5z + 0.4 - 0.3i.
\end{align*}
The mappings $f_0, g_0, h_0$ map $X(0)$ into itself and are contractions on $(X(0), d)$ with Lipschitz constant $0.5$. Thus, $\mathcal{I}_0$ is a classical IFS consisting of contractions on $(X(0), d)$.

Defining $\tau_t = 0.8i\sin t$, the temporal IFS $\mathcal{I}_t = \{f_t, g_t, h_t\}$ acting on the fiber $X(t)$ evaluates to:
\begin{align*}
    f_t(z) &= \big(\Phi_{t, 0} \circ f_0 \circ \Phi_{0, t}\big)(z) = \frac{(0.5 + 0.4i\tau_t)z + 0.4i}{(0.5\tau_t - 0.4i\tau_t^2)z + (1 - 0.4i\tau_t)}, \\
    g_t(z) &= \big(\Phi_{t, 0} \circ g_0 \circ \Phi_{0, t}\big)(z) = \frac{(0.5 - (0.4 + 0.3i)\tau_t)z - (0.4 + 0.3i)}{(0.5\tau_t + (0.4 + 0.3i)\tau_t^2)z + (1 + (0.4 + 0.3i)\tau_t)}, \\
    h_t(z) &= \big(\Phi_{t, 0} \circ h_0 \circ \Phi_{0, t}\big)(z) = \frac{(0.5 + (0.4 - 0.3i)\tau_t)z + (0.4 - 0.3i)}{(0.5\tau_t - (0.4 - 0.3i)\tau_t^2)z + (1 - (0.4 - 0.3i)\tau_t)}.
\end{align*}
Because the evolution family consists of Möbius transformations whose poles lie outside $X(t)$, and the reference mappings are affine, the induced mappings are compositions of analytic functions. Therefore, they remain analytic on $X(t)$, and their contractivity factors evaluate exactly to the maximum modulus of their derivatives over the domain. 

As time evolves, these moduli fluctuate. At specific times, the maximum modulus exceeds $1$, and contractivity fails.

By Theorem \ref{thm:topological_stability}, the corresponding system $\mathcal{I}_t$ of the temporal iterated function system possesses a unique, compact attractor $\mathcal{A}^*(t) = \Phi_{t, 0}(A_0)$ for all $t \in T$. Figure \ref{fig:conformal_flow} illustrates instances of $\mathcal{A}^*(t)$ at $t=0$, $1$, and $4$.

\begin{figure}[htbp]
    \centering
    \includegraphics[width=\textwidth]{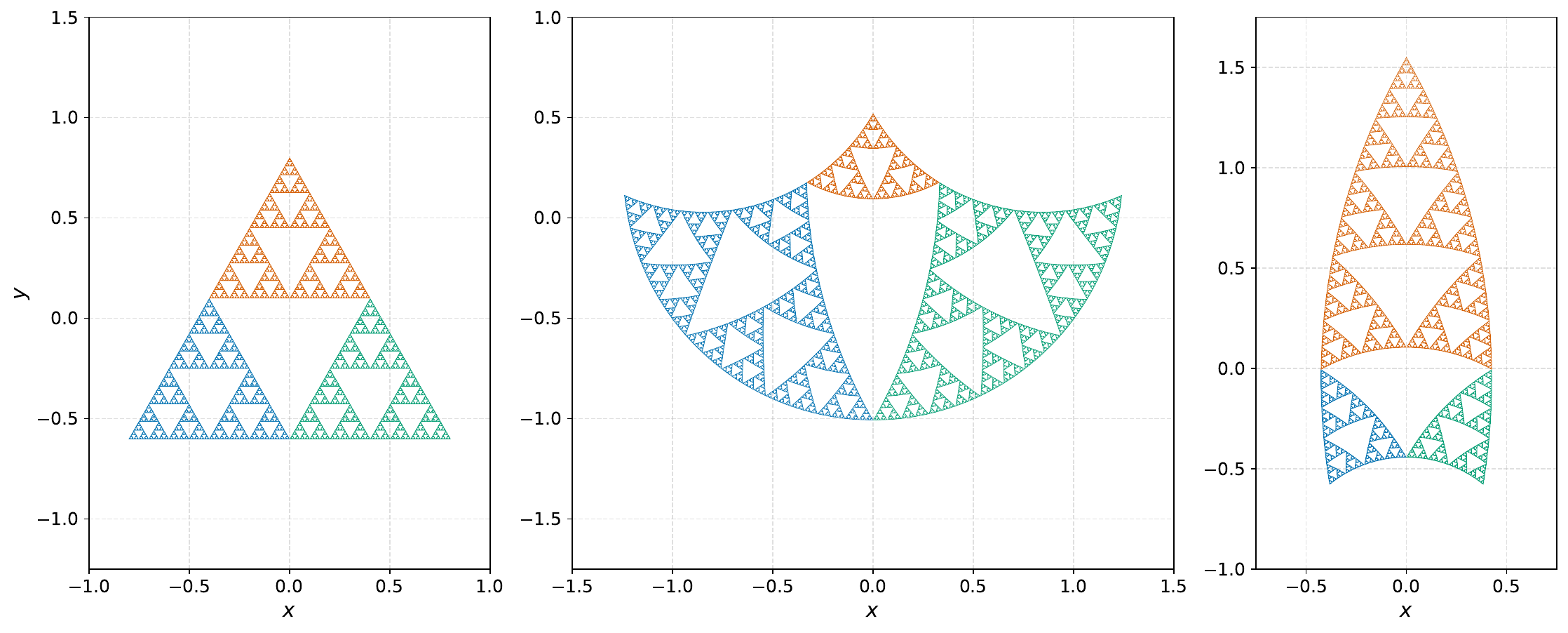}
    \caption{Evolution of the attractor associated with Example \ref{ex:conformal_flow}, shown from left to right at times $t=0$, $1$, and $4$. The red, blue, and green portions denote the images of the attractor under the mappings $f_t$, $g_t$, and $h_t$, respectively. The flow-induced iterated function system consists of contractions at $t=0$, but fails to satisfy contractivity at $t=1$ and $t=4$. The attractor persists despite the failure of classical contraction conditions.}
    \label{fig:conformal_flow}
\end{figure}
\end{example}

\section{Conclusion}
In this paper, we introduced temporal metric spaces to analyze structures whose underlying domains evolve over time. By constructing interval-dependent metrics on trajectory spaces, we established conditions for the inheritance of completeness from base spaces to trajectory spaces and their associated temporal hyperspaces. We applied this framework to iterated function systems, proving that flow-induced temporal attractors persist even when ambient contractivity fails. The extension of iterated function systems serves as one specific application of this structure. The framework of temporal metric spaces provides a basis to investigate how other fundamental concepts from classical metric space theory adapt and develop in time-dependent environments.

\section*{Acknowledgements}
This research was financially supported by the University Grants Commission (UGC), Government of India, through the award of the Junior Research Fellowship (JRF) (Ref. No. [1090(CSIR-UGC NET JUNE 2019)]) awarded to the first author. The first and third authors acknowledge the institutional support provided by Rajagiri School of Engineering and Technology, Kerala, India, and the second author acknowledges the institutional support provided by Muthoot Institute of Technology and Science, Kerala, India.

\bibliographystyle{amsplain}
\bibliography{references}

\end{document}